\documentclass[12pt]{amsart}
\usepackage{amssymb, amsmath, amsfonts, amsthm, esint,graphicx,enumerate}
\usepackage[abbrev]{amsrefs}
\usepackage[margin=0.8 in]{geometry}
\usepackage[colorlinks=true, linkcolor=blue]{hyperref}

\makeatletter
\@namedef{subjclassname@2020}{\textup{2020} Mathematics Subject Classification}
\makeatother

\title[Integral Ricci Curvature for Graphs]{Integral Ricci Curvature for Graphs}

\author{Xavier Ramos Oliv\'e}
\thanks{This work was funded in part by the National Science Foundation under Grant Number 1916439, through the participation in the \emph{2023 MRC on Ricci Curvatures of Graphs and Applications to Data Science}.}
\address{Department of Mathematics and Computer Science\\
College of the Holy Cross\\
Swords Hall 326\\
1 College St\\
Worcester, MA 01610}
\email{\href{mailto:xramos-olive@holycross.edu}{xramos-olive@holycross.edu}}

\date{}

\theoremstyle{definition} 
\newtheorem{definition}{Definition}[section]

\newtheorem{example}{Example}[section]
\newtheorem{lemma}{Lemma}[section]
\newtheorem{remark}{Remark}[section]

\newtheorem{theorem}{Theorem}[section]

\newtheorem*{theorem*}{Theorem}
\begin{document}
\maketitle

\begin{abstract}
We introduce the notion of integral Ricci curvature  $I_{\kappa_0}$ for graphs, which measures the amount of Ricci curvature below a given threshold $\kappa_0$. We focus our attention on the Lin-Lu-Yau Ricci curvature. As applications, we prove a Bonnet-Myers-type diameter estimate, a Moore-type estimate on the number of vertices of a graph in terms of the maximum degree $d_M$ and diameter $D$, and a Lichnerowicz-type estimate for the first eigenvalue $\lambda_1$ of the Graph Laplacian, generalizing the results obtained by Lin, Lu, and Yau. All estimates are uniform, depending only on geometric parameters like $\kappa_0$, $I_{\kappa_0}$, $d_M$, or $D$, and do not require the graphs to be positively curved.
\end{abstract}

\section{Introduction}
\subsection{Curvature in the Smooth Setting}
\vspace{1em}

Ricci curvature plays a central role in Geometric Analysis. The Bonnet-Myers estimate \cite{Myers} establishes that an $n$-dimensional Riemannian manifold satisfying ${\rm Ric}\geq (n-1)K>0$ has diameter bounded above by
\[{\rm Diam}\leq \frac{\pi}{\sqrt{K}}.\]
Lichnerowicz \cite{Lichnerowicz} showed that the first nonzero eigenvalue $\lambda_1$ of the Laplace-Beltrami operator satisfies
\[\lambda_1\geq nK.\]
There exists very extensive literature of results using pointwise lower bounds on the Ricci curvature to prove volume comparison results, numerous kinds of eigenvalue estimates, spectral gap estimates, estimates on the Betti numbers, Sobolev inequalities, estimates on the heat kernel, Harnack inequalities, etc. The strength of these results is that they provide us with estimates that hold uniformly for all manifolds whose Ricci curvature is bounded below by $(n-1)K$. We refer the interested reader to the survey \cite{DaiWei} or the book \cite{li-book} and the references therein.

In \cite{Gallot}, Gallot introduced the notion of integral curvature $\overline{k}(p,K)$ on $n$-dimensional Riemannian manifolds, which measures the amount of Ricci curvature below the threshold $(n-1)K$ in an $L^p$ sense. Petersen and Wei \cite{PetersenWei} generalized the Bishop-Gromov volume comparison results to integral curvature conditions, enabling the use of integral curvature assumptions to prove similar estimates to the ones mentioned above, allowing the curvature to dip below the threshold $(n-1)K$. Integral curvature conditions are less sensitive to small perturbations of the geometry than pointwise lower bounds on the curvature, providing estimates that apply uniformly to a much larger class of manifolds. Two examples of such estimates are due to Aubry \cite{Aubry}, who generalized the Bonnet-Myers and the Lichnerowicz estimates, showing that for $K>0$ and $p>n/2$, we have
\begin{equation}\label{aubry_diam}
{\rm Diam}\leq \frac{\pi}{\sqrt{K}}[1+C(p,n)\overline{k}(p,K)]    
\end{equation}
and
\begin{equation}\label{aubry_eigen}
    \lambda_1 \geq nK[1-C(p,n)\overline{k}(p,K)].
\end{equation}
These estimates are sharp in the following sense. In the limit where ${\rm Ric}\geq (n-1)K>0$ we have that $\overline{k}(p,K)=0$, therefore Aubry's estimates recover the classical results of Bonnet-Myers and Lichnerowicz when the curvature is bounded below pointwise. The quantity $\overline{k}(p,K)$ acts as an error term. These estimates allow the curvature to be significantly below the threshold $(n-1)K$, it can even be negative. As long as this does not happen in a set of large measure (as long as $\overline{k}(p,K)$ is small enough), we can obtain meaningful estimates that are a perturbation of the pointwise ones. It is well known that smallness of $\overline{k}(p,K)$ is a necessary condition (see e.g. \cite{AndersonRamosSpinelli} or \cite{Gallot}).

Integral curvature conditions have been used to prove Sobolev inequalities (\cite{DaiWeiZhang},  \cite{Gallot}, \cite{PetersenSprouse}, \cite{rose2}), eigenvalue and spectral gap estimates (\cite{Aubry}, \cite{Gallot}, \cite{RamosRoseWangWei}, \cite{RSWZ}), and estimates for the solutions to elliptic and parabolic equations (\cite{Azami}, \cite{PostRamosRose}, \cite{ramos}, \cite{rose2}, \cite{Wang2}, \cite{WangWang}, \cite{ZhangZhu1}, \cite{ZhangZhu2}), among many other results. Recently, these conditions have been generalized further into different directions, such as the Kato condition (\cite{Carron}, \cite{CarronRose}, \cite{rose}, \cite{rose3}, \cite{RoseStollmann}, \cite{RoseWei}) and integral curvature conditions on the Bakry-\'Emery Ricci curvature on smooth metric measure spaces (\cite{Mi}, \cite{RamosSeto}, \cite{Tadano}, \cite{WangWei},\cite{Wu1}, \cite{Wu2}).

\subsection{Curvature in the Discrete Setting}
\vspace{1em}

Since Ricci curvature plays such a central role in Geometric Analysis, and particularly in the study of the Laplacian and its eigenvalues, it is reasonable to expect that generalized notions of Ricci curvature on graphs would be very valuable in the study of Spectral Graph Theory. Along those lines, Chung and Yau introduced the notion of Ricci flat graphs in \cite{ChungYau} and studied Harnack Inequalities on these spaces, translating well known results from the smooth setting to the discrete setting. The development of synthetic Ricci curvature notions on geodesic metric spaces, which started with the work of Sturm and Lott-Villani (\cite{LottVillani}, \cite{Sturm1},\cite{Sturm2}), opened up the possibility to generalize these classical estimates from the smooth manifold setting to the setting of metric spaces.  Ollivier \cite{Ollivier} introduced the Ollivier Ricci curvature (see \eqref{Ollivier_curvature}), which is defined for Markov chains on metric spaces. In particular, this notion of curvature can be applied to a graph $G$. 

Throughout this article, we assume that $G=(V,E)$ is a connected simple graph, that is to say, it is an unweighted, undirected graph containing no graph loops or multiple edges ($E\subseteq V\times V$). If $\{x,y\}\in E$ for some $x,y\in V$, we will denote the edge from $x$ to $y$ by $xy\in E$.  We will also assume $V$ to be finite for simplicity, although this is not strictly necessary. We will study $(G,d)$ as a length metric space, where the distance between $x,y\in V$, $d(x,y)$, is taken to be the length of the shortest path between $x$ and $y$. Then, the diameter of $G$ is the largest distance between two vertices, ${\rm Diam}(G)=\displaystyle\max_{x,y\in V} d(x,y)$.

For any $\alpha\in [0,1]$, using the $\alpha-$lazy random walk on $G$ (see the definition in \eqref{alpha_lazy}) as a Markov process, the Ollivier Ricci curvature defines the so-called $\alpha$-Ricci curvature $\kappa_\alpha(x,y)$ between any two nodes $x,y\in V$  (see \eqref{alpha_Ricci_curvature}). In \cite{LinLuYau}, Lin-Lu-Yau introduced a limiting version of the $\alpha$-Ricci curvature,
\[\kappa_{LLY}(x,y)=\lim_{\alpha\rightarrow 1^{-}} \frac{\kappa_\alpha (x,y)}{1-\alpha},\]
which is known as the Lin-Lu-Yau Ricci curvature $\kappa_{LLY}(x,y)$. This notion of curvature presents a few advantages compared to the $\alpha$-Ricci curvature, like being well behaved with products of graphs. This is the notion of Ricci curvature on graphs that we will focus on in this article. A very interesting result from M\"unch-Wojciechowski \cite{MunchWojciechowski} shows that there exists an equivalent way of defining the Lin-Lu-Yau Ricci curvature which is limit-free. A precise definition of $\kappa_\alpha(x,y)$ and $\kappa_{LLY}(x,y)$ can be  found in Section~\ref{sec_Background}.

\subsection{Main Results}
\vspace{1em}

Among other results, \cite{LinLuYau} proves a Bonnet-Myers'-type diameter estimate, a Moore-type bound on the number of vertices $|V|$, and a Lichnerowicz-type eigenvalue estimate for positively curved graphs, i.e. graphs for which $\kappa_{LLY}(x,y)\geq \kappa_0>0$ for all edges $xy\in E$; see their precise statements in Theorem~\ref{LLY_BonnetMyers}, Theorem~\ref{LLY_Moore}, and Theorem~\ref{LLY_Lichnerowicz}. Our goal is to prove generalizations of these three results for graphs that are not necessarily positively curved, using instead an integral curvature quantity $I_{\kappa_0}$ (or $I^\alpha_{\kappa_0}$), which measures the amount of Lin-Lu-Yau Ricci curvature (or $\alpha$-Ricci curvature, respectively) below a given threshold $\kappa_0$ (or $(1-\alpha)\kappa_0$, respectively). The quantities $I_{\kappa_0}$ and $I^\alpha_{\kappa_0}$ are defined in such a way that $I_{\kappa_0}=0$ if and only if $\kappa_{LLY}(x,y)\geq \kappa_0$ for all $xy\in E$, and $I_{\kappa_0}^\alpha = 0$  if and only if $\kappa_\alpha(x,y)\geq \kappa_0$ for all $xy\in E$. This way, our estimates recover the results in \cite{LinLuYau} when $I_{\kappa_0}=0$ (respectively $I_{\kappa_0}^\alpha =0$). The precise definition of $I_{\kappa_0}$ and $I^\alpha_{\kappa_0}$ is given in Section~\ref{sec_intcurv} (see Definition~\ref{integralcurv_def}). In particular, our main results are as follows. 

First, we prove a Myers' type diameter estimate for graphs that are not necessarily positively curved, extending Theorem~\ref{LLY_BonnetMyers} to integral curvature conditions in the spirit of Aubry's result \eqref{aubry_diam}.

\begin{theorem}\label{BonnetMyers_intcurv}
    For any $\kappa_0>0$ and $\alpha\in [0,1)$, the diameter of $G$ can be bounded by 
    \begin{equation}\label{alpha_diam_bound}
        {\rm Diam}(G) \leq \left\lfloor\frac{2+ \displaystyle \frac{I^\alpha_{\kappa_0}}{(1-\alpha)}}{\kappa_0}\right\rfloor
    \end{equation}
    and 
    \begin{equation}\label{LLY_diam_bound}
        {\rm Diam}(G) \leq \left\lfloor\frac{2+ I_{\kappa_0}}{\kappa_0}\right\rfloor. 
    \end{equation}
\end{theorem}

We show in Example~\ref{example_sharp_diam} that this diameter estimate is sharp for Path graphs, which are graphs where $\displaystyle\min_{xy\in E}\ \kappa_{LLY}(x,y) =0$ and therefore Theorem~\ref{LLY_BonnetMyers} would not apply.

Our second result is a Moore-type estimate on the number of vertices $n=|V|$ of a graph with prescribed diameter $D$ and maximum degree $d_M$ (see \eqref{Moore} for the original estimate  Moore was interested in, which can be found in \cite{HoffmanSingleton}).  Our estimate generalizes Theorem \ref{LLY_Moore} to integral curvature conditions, allowing our graphs to not be positively curved.

\begin{theorem}\label{Moore_intcurv}
Let $d_M$ be the maximum degree of $G$, and suppose that ${\rm Diam}(G)= D$. For any $\kappa_0\in \mathbb{R}$, the number of vertices $n=|V|$ of $G$ is at most

\begin{equation}\label{Moore_intcurv_bound_diam}
    n\leq 1+ \sum_{k=1}^{D} (d_M)^k \prod_{i=1}^{k-1}\left[1+ \frac{I_{\kappa_0}-i\kappa_0}{2}\right].
\end{equation}
Moreover, if $\kappa_0>0$,  we have
\begin{equation}\label{Moore_intcurv_bound}
    n\leq 1+ \sum_{k=1}^{\left\lfloor \frac{2+I_{\kappa_0}}{\kappa_0}\right\rfloor} (d_M)^k \prod_{i=1}^{k-1}\left[1+ \frac{I_{\kappa_0}-i\kappa_0}{2}\right].
\end{equation}
\end{theorem}
\begin{remark}
    Note that for any $i<{\rm Diam}(G)$ we have 
    \[1+\frac{I_{\kappa_0}-i\kappa_0}{2}>0.\]
     Since $I_{\kappa_0}\geq 0$, this is clear when $\kappa_0\leq 0$, and it follows from Theorem~\ref{BonnetMyers_intcurv}, \eqref{LLY_diam_bound}, when $\kappa_0>0$.  
\end{remark}


As a corollary of Theorem~\ref{Moore_intcurv}, one can derive an obstruction on the curvature of a hypothetical Moore graph with ${\rm Diam}(G)=2$ and $d_M=57$, whose existence, to the best of our knowledge, is still unknown  (see the survey \cite{Dalfo}, as well as \cite{BannaiIto}, \cite{Damerell}, \cite{FaberKeegan}). In particular, we observe in Remark~\ref{remark_Moore_graphs} that such graph would have to satisfy 
    \[\kappa_0-I_{\kappa_0}\leq \frac{2}{57},\]
and in particular it would need to have at least one edge satisfying 
 \[\kappa_{LLY}(x,y)\leq \frac{2}{57}.\]
 Unfortunately, this is not a new obstruction, since it follows from recent work in \cite{LinYou} that such a graph would have at least one edge with non-positive curvature. We discuss this obstruction in more detail in Remark~\ref{remark_Moore_graphs}.

Our last result is a Lichnerowicz-type estimate in the spirit of Aubry's result \eqref{aubry_eigen}. It is an estimate like  Theorem~\ref{LLY_Lichnerowicz}, where we prove a lower bound for the first non-zero eigenvalue $\lambda_1$ of the Graph Laplacian in terms of the integral curvature of $G$. Here by Graph Laplacian we mean what is referred to sometimes as the normalized random walk Laplacian $\Delta = I-D^{-1}A$, where $I$ is the identity matrix, $A$ denotes the adjacency matrix of $G$, and $D$ is the diagonal matrix of degrees of the vertices in $G$. 

\begin{theorem}\label{Lichnerowicz_intcurv}
For any $\kappa_0>0$ and any $\alpha \in [0,1)$, we have
\begin{equation}\label{Lic_int_alpharicci}
    \lambda_1 \geq \kappa_0 -\frac{I_{\kappa_0}^\alpha}{1-\alpha},
\end{equation}
and
\begin{equation}\label{Lic_int_LLYricci}
    \lambda_1 \geq \kappa_0 - I_{\kappa_0}.
\end{equation}
\end{theorem}

Unfortunately, due to our specific definition of $I_{\kappa_0}$ and $I^\alpha_{\kappa_0}$ and the discrete nature of graphs, although Theorem~\ref{Lichnerowicz_intcurv} recovers Theorem~\ref{LLY_Lichnerowicz} for positively curved graphs, it does not  prove anything beyond Theorem~\ref{LLY_Lichnerowicz}. It is, in a sense, just a reformulation of the result from \cite{LinLuYau} in terms of integral curvature. We discuss this issue in more detail in Remark~\ref{remark_Lichnerowicz}.  Despite this, we decided to include this result here as a potential avenue for future work.  

 To the best of our knowledge, this is the first time integral curvature conditions have been defined for the $\alpha-$Ricci and the Lin-Lu-Yau Ricci curvatures of a graph in the literature, and used to prove results for graphs that are not positively curved. We must point out that there exists previous work studying graphs with non-negative (pointwise) Lin-Lu-Yau Ricci curvature, as opposed to positive, as is the case of \cite{Munch}. Also, in \cite{Munch3}, M\"unch considered a weighted average of the Ollivier Ricci curvatures, similar in spirit to integral Ricci curvature; in that case, the author established a relationship between the average curvature, the average degree, and the average distance of the graph. In the context of the Bakry-\'Emery Ricci curvature for graphs, which is not equivalent to Lin-Lu-Yau Ricci curvature, M\"unch-Rose \cite{MunchRose}  considered a Kato-type curvature condition, which is more general than integral curvature assumptions. The authors of \cite{HuaMunch}, \cite{LMPR}, and \cite{Munch2} also studied graphs that were only non-negatively curved outside an exceptional set. On a separate note, the question of which graphs satisfy equality for Theorem~\ref{LLY_BonnetMyers} and Theorem~\ref{LLY_Lichnerowicz} was explored in \cite{CKKLMP}, who showed that in the discrete setting there is no rigidity, as many graphs achieve the optimal bounds.

 This article has been structured in the following way. In Section~\ref{sec_Background}, we introduce all the necessary background, including the definitions of Ollivier and Lin-Lu-Yau Ricci curvature, as well as the statements of the theorems we are generalizing from \cite{LinLuYau}. In Section~\ref{sec_intcurv} we formally define $I_{\kappa_0}$ and $I_{\kappa_0}^\alpha$. Finally, in Section~\ref{sec_mainthm}, we prove our main theorems and provide some examples.

\section*{Acknowledgements}

The author would like to thank Christian Rose and Florentin M\"unch for their valuable feedback on an earlier version of this article. The author would also like to thank Fan Chung, Mark Kempton, Wuchen Li, Linyuan Lu, Zhiyu Wang, and the American Mathematical Society for organizing the \emph{2023 Mathematics Research Community on Ricci Curvatures of Graphs and Applications to Data Science}, which inspired this work.

\section{Background}\label{sec_Background}
In the following, we recall the key definitions needed to define the Ollivier, the $\alpha$-Ricci, and the Lin-Lu-Yau Ricci curvatures. We will follow a similar notation as in \cite{LinLuYau} and \cite{Ollivier}.
 
Given two probability distributions $m_1$ and $m_2$ over $V$, a coupling between $m_1$ and $m_2$ is a mapping $A:V\times V\rightarrow [0,1]$ satisfying 
\begin{equation}\label{marginals}\sum_{y\in V}A(x,y)=m_1(x) \text{ and }\sum_{x\in V}A(x,y)=m_2(y).\end{equation}
We will denote $d(x,y)$ the graph distance $d:V\times V\rightarrow \mathbb{N}$. The transportation distance between $m_1$  and $m_2$ is defined as 
\begin{equation}\label{transp_distance}
    W(m_1,m_2)=\inf_{A}\sum_{x,y\in V}A(x,y)d(x,y),
\end{equation}
where the infimum is taken over all couplings $A$ between $m_1$ and $m_2$. Finding $W(m_1,m_2)$ requires solving a linear optimization problem with constraints given by \eqref{marginals}. By the duality theorem of linear optimization, the transportation distance can be rewritten as 
\begin{equation}\label{transp_distance_dual}
    W(m_1,m_2)=\sup_{f}\sum_{x\in V}f(x)[m_1(x)-m_2(x)],
\end{equation}
where the supremum is taken over all $1$-Lipschitz functions $f:V\rightarrow \mathbb{R}$. Let $(m_x)_{x\in V}$ be a family of probability measures over $V$ indexed by $V$. The transportation distance is the key in defining the Ollivier Ricci curvature $\kappa_{Oll}$
\begin{equation}\label{Ollivier_curvature}
    \kappa_{Oll}(x,y):= 1-\frac{W(m_x,m_y)}{d(x,y)}.
\end{equation}
\begin{remark}
    Note that the curvature is defined for any pair of vertices $x,y\in V$, $x\not=y$, not necessarily adjacent.
\end{remark}


Different choices for the measures $(m_x)_{x\in V}$ lead to different notions of Ricci curvature. For instance, the $\alpha$-Ricci curvature can be obtained from the probability distribution of the $\alpha$-lazy random walk as follows. For any $\alpha \in [0,1)$ and any vertex $x\in V$, define
\begin{equation}\label{alpha_lazy}
    m_x^\alpha (v) :=\begin{cases} \alpha & \text{if }v=x,\\ \displaystyle \frac{1-\alpha}{d_x} & \text{if }v\in \Gamma(x),\\ 0&\text{if }v\not \in \Gamma(x),\end{cases}
\end{equation}
where $\Gamma(x) =\{v:\ vx\in E\}$ is the set of vertices adjacent to $x$ and $d_x=|\Gamma(x)|$ is the degree of $x$. Then, the $\alpha$-Ricci curvature is defined as 
\begin{equation}\label{alpha_Ricci_curvature}
\kappa_\alpha(x,y):=1-\frac{W(m_x^\alpha,m_y^\alpha)}{d(x,y)}.   
\end{equation}
In \cite{LinLuYau}, Lin-Lu-Yau showed that for fixed $x,y\in V$ the function 
\[
    h_{xy}(\alpha)=\frac{\kappa_\alpha(x,y)}{1-\alpha}
\]
is increasing in $\alpha$ and bounded. Therefore, one can define the Lin-Lu-Yau Ricci curvature as the limit
\begin{equation}\label{LinLuYau_curvature}
\kappa_{LLY}(x,y):=\lim_{\alpha\rightarrow 1^-} \frac{\kappa_\alpha(x,y)}{1-\alpha}.\end{equation}
One of the advantages of the Lin-Lu-Yau Ricci curvature is that it behaves well with the Cartesian product of graphs (see \cite{LinLuYau}~Theorem 3.1). Moreover, the authors are able to prove the following three results for graphs with positive Lin-Lu-Yau Ricci curvature, which we will extend to integral curvature assumptions in this article.

\begin{theorem}[Bonnet-Myers-type estimate, \cite{LinLuYau}~Thm 4.1]\label{LLY_BonnetMyers}
    For any $x,y\in V$, if $\kappa_{LLY}(x,y)>0$, then
    \begin{equation}\label{local_diameter_estimate}
        d(x,y)\leq \left\lfloor\frac{2}{\kappa_{LLY}(x,y)}\right\rfloor.
    \end{equation}
    Moreover, if for any edge $xy\in E$ we have that $\kappa_{LLY}(x,y)\geq \kappa_0>0$, then the diameter of $G$ is bounded as
    \begin{equation}\label{global_diameter_estimate}
        {\rm diam}(G) \leq \left\lfloor\frac{2}{\kappa_0}\right\rfloor.
    \end{equation}
\end{theorem}
\begin{theorem}[Moore-type bound, \cite{LinLuYau}~Thm 4.3]\label{LLY_Moore}
    Suppose that for any $xy\in E$ we have $\kappa_{LLY}(x,y)\geq \kappa_0>0$. Let $d_{M}$ be the maximum degree of $G$. Then if $n=|V|$ denotes the number of vertices, we have
    \begin{equation}
        n\leq 1+ \sum_{k=1}^{\lfloor 2/\kappa_0 \rfloor}(d_{M})^k \prod_{i=1}^{k-1}\left( 1-i\frac{\kappa_0}{2} \right).
    \end{equation}
\end{theorem}
\begin{theorem}[Lichnerowicz-type estimate, \cite{LinLuYau}~Thm 4.2]\label{LLY_Lichnerowicz}
    Let $\lambda_1>0$ denote the first non-zero eigenvalue of the positive graph Laplacian $\Delta = I-D^{-1}A$, where $A$ denotes the adjacency matrix of  $G$ and $D$ is the diagonal matrix of degrees. Suppose that for any edge $xy\in E$, we have $\kappa_{LLY}(x,y)\geq \kappa_0>0$. Then 
    \begin{equation}\label{Lichnerowicz_estimate}
        \lambda_1 \geq \kappa_0.
    \end{equation}
\end{theorem}
\section{Integral Ricci Curvature}\label{sec_intcurv}
\begin{definition}\label{integralcurv_def}
 For any $\kappa_0\in \mathbb{R}$ and $\alpha\in [0,1)$, let the functions $\rho_{\kappa_0}^\alpha : V\times V\rightarrow \mathbb{R}$ and $\rho_{\kappa_0} : V\times V\rightarrow \mathbb{R}$ be defined as 
 \begin{equation}
     \rho_{\kappa_0}^\alpha (x,y) :=\max \{0,(1-\alpha)\kappa_0-\kappa_\alpha(x,y)\}
 \end{equation}
 and 
 \begin{equation}
     \rho_{\kappa_0} (x,y) :=\max \{0,\kappa_0-\kappa_{LLY}(x,y)\}.
 \end{equation}
 We define the Integral $\alpha$-Ricci curvature $I^\alpha_{\kappa_0}$ and the Integral Lin-Lu-Yau Ricci curvature $I_{\kappa_0}$ to be the following quantities
 \begin{equation}\label{alphaintegralcurv_def}
 I^\alpha_{\kappa_0}:= \displaystyle\sum_{xy\in E}\rho^\alpha_{\kappa_0}(x,y)
 \end{equation}
 and 
  \begin{equation}\label{LLYintegralcurv_def}
 I_{\kappa_0}:= \sum_{xy\in E}\rho_{\kappa_0}(x,y).
 \end{equation}
\end{definition}

\begin{remark}
    The functions $\rho_{\kappa_0}^\alpha(x,y)$ and $\rho_{\kappa_0}(x,y)$ measure the amount by which the inequalities $$\kappa_\alpha(x,y)\geq (1-\alpha) \kappa_0\ \ \ \text{ and }\ \ \ \kappa_{LLY}(x,y)\geq \kappa_0,$$ respectively, fail to be satisfied. That is to say, the inequalities are satisfied if and only if $\rho_{\kappa_0}^\alpha (x,y) =0$ or $\rho_{\kappa_0} (x,y) =0$, respectively. Otherwise, $\rho_{\kappa_0}^\alpha(x,y)$ measures the amount of $\alpha$-Ricci curvature between the vertices $x$ and $y$ below the threshold $(1-\alpha)\kappa_0$, and similarly $\rho_{\kappa_0}(x,y)$ measures the amount of Lin-Lu-Yau Ricci curvature  between $x$ and $y$ below the threshold $\kappa_0$. 
\end{remark}
\begin{remark}
    The  quantity $I^{\alpha}_{\kappa_0}$ measures the total amount of $\alpha$-Ricci curvature below the threshold $(1-\alpha)\kappa_0$ over all the edges $xy\in E$ in the graph. By definition it follows that $I^\alpha_{\kappa_0}\geq 0$, and we have that $$I^\alpha_{\kappa_0}=0\iff \kappa_{\alpha}(x,y)\geq (1-\alpha)\kappa_0 \text{\ for all } xy\in E.$$ 

Similarly, the quantity $I_{\kappa_0}$ measures the total amount of Lin-Lu-Yau Ricci curvature below the threshold $\kappa_0$ over all the edges $xy\in E$ in the graph. By definition it follows that $I_{\kappa_0}\geq 0$, and we have that $$I_{\kappa_0}=0\iff \kappa_{LLY}(x,y)\geq \kappa_0 \text{\ for all } xy\in E.$$ 
\end{remark}

\begin{remark}
    Note that the definitions above have been chosen so that
    \[\lim_{\alpha\rightarrow 1^-}\frac{I^\alpha_{\kappa_0} }{1-\alpha}= I_{\kappa_0}.\]
\end{remark}


\section{Main Results}\label{sec_mainthm}
The following lemma will be crucial in our subsequent discussions, and can be regarded as one of the key contributions of this article. The first half of the proof of the lemma follows the proof of \cite{LinLuYau}~Lemma 2.3, but afterwards our proofs differ. We will include the full proof for convenience.
\begin{lemma}\label{keylemma}
    Let $x,y\in V$ be two (not necessarily adjacent) vertices, and consider a minimizing path $P$ from $x$ to $y$, $P=(x_0, x_1, x_2, \ldots, x_{d(x,y)})$, where $x_0=x$, $x_{d(x,y)}=y$, and $x_{i-1}$ is adjacent to $x_i$.  For any $\kappa_0\in \mathbb{R}$ and any $\alpha\in [0,1)$, we have 
    \begin{equation}\label{lemma_alpha}
        \kappa_{\alpha}(x,y)\geq \frac{\displaystyle\sum_{i=1}^{d(x,y)}\kappa_\alpha(x_{i-1},x_i)}{d(x,y)}\geq (1-\alpha)\kappa_0-\frac{I^\alpha_{\kappa_0}}{d(x,y)}.
    \end{equation}
    and
    \begin{equation}\label{lemma_LLY}
        \kappa_{LLY}(x,y) \geq \frac{\displaystyle\sum_{i=1}^{d(x,y)}\kappa_{LLY}(x_{i-1},x_i)}{d(x,y)} \geq \kappa_0 -\frac{I_{\kappa_0}}{d(x,y)}.
    \end{equation}
\end{lemma}
\begin{proof}
By definition of $\kappa_\alpha$ and using the triangle inequality of the transportation distance, we have
\begin{align*}
    \kappa_\alpha(x,y) &= 1-\frac{W(m_x^\alpha,m_y^\alpha)}{d(x,y)} \\
    &\geq 1-\frac{\displaystyle\sum_{i=1}^{d(x,y)} W(m^\alpha_{x_{i-1}},m^\alpha_{x_i})}{d(x,y)}\\
    &=\frac{\displaystyle\sum_{i=1}^{d(x,y)}\kappa_{\alpha}(x_{i-1},x_i)}{d(x,y)},
\end{align*}
which gives us the first inequality in~\eqref{lemma_alpha}. By adding and subtracting $(1-\alpha)\kappa_0$ in each summand, and using the fact that $$\rho_{\kappa_0}^\alpha(x_{i-1},x_i) \geq (1-\alpha)\kappa_0 - \kappa_\alpha(x_{i-1},x_i),$$ we get
\begin{align*}
    \frac{\displaystyle\sum_{i=1}^{d(x,y)}\kappa_{\alpha}(x_{i-1},x_i)}{d(x,y)} &=     \frac{\displaystyle\sum_{i=1}^{d(x,y)}\left[\kappa_{\alpha}(x_{i-1},x_i) - (1-\alpha)\kappa_0\right]}{d(x,y)} + (1-\alpha)\kappa_0\\
    & \geq \frac{\displaystyle\sum_{i=1}^{d(x,y)}-\rho^\alpha_{\kappa_0}(x_{i-1},x_i)}{d(x,y)} + (1-\alpha)\kappa_0\\
    & \geq \frac{\displaystyle\sum_{uv\in E}-\rho^\alpha_{\kappa_0}(u,v)}{d(x,y)} + (1-\alpha)\kappa_0\\
    &=(1-\alpha)\kappa_0-\frac{I^\alpha_{\kappa_0}}{d(x,y)},
\end{align*}
which proves~\eqref{lemma_alpha}. Dividing~\eqref{lemma_alpha} by $(1-\alpha)$ and taking the limit as $\alpha\rightarrow 1^-$ we obtain~\eqref{lemma_LLY}.
\end{proof}
\begin{remark}
    Note that in the case where the curvature is bounded below pointwise (i.e. $I_{\kappa_0}=0$ or $I^\alpha_{\kappa_0}=0$), this lemma recovers Lemma 2.3 from \cite{LinLuYau}.
\end{remark}
\subsection{Bonnet-Myers-type estimate}
In this subsection we will generalize Theorem~\ref{LLY_BonnetMyers} to the integral curvature setting. To that end, we need to recall the following lemma from \cite{LinLuYau}. 

\begin{lemma}[\cite{LinLuYau}~Lemma 2.2]
    For any $\alpha\in [0,1]$ and any two vertices $x$ and $y$, we have
    \begin{equation}\label{LLY_lemma2.2}
        \kappa_\alpha(x,y)\leq (1-\alpha)\frac{2}{d(x,y)}.
    \end{equation}
\end{lemma}

Now we are ready to prove Theorem~\ref{BonnetMyers_intcurv}, which we restate below for convenience.
\begin{theorem*}[\bf\ref{BonnetMyers_intcurv}]
    For any $\kappa_0>0$ and $\alpha\in [0,1)$, the diameter of $G$ can be bounded by 
    \begin{equation}\label{alpha_diam_bound}
        {\rm Diam}(G) \leq \left\lfloor\frac{2+ \displaystyle \frac{I^\alpha_{\kappa_0}}{(1-\alpha)}}{\kappa_0}\right\rfloor
    \end{equation}
    and 
    \begin{equation}\label{LLY_diam_bound}
        {\rm Diam}(G) \leq \left\lfloor\frac{2+ I_{\kappa_0}}{\kappa_0}\right\rfloor. 
    \end{equation}
\end{theorem*}
\begin{proof}[Proof of Theorem~\ref{BonnetMyers_intcurv}]
    Let $x,y\in V$ be any two vertices. Combining \eqref{lemma_alpha} with \eqref{LLY_lemma2.2}, we get 
    \[(1-\alpha)\frac{2}{d(x,y)}\geq \kappa_\alpha(x,y) \geq (1-\alpha)\kappa_0-\frac{I^\alpha_{\kappa_0}}{d(x,y)}.\]
    Therefore, 
    \[d(x,y) \leq \frac{2+ \displaystyle \frac{I^\alpha_{\kappa_0}}{(1-\alpha)}}{\kappa_0}.\]
    Since $d(x,y)$ is an integer, we have
        \[d(x,y) \leq \left\lfloor\frac{2+ \displaystyle \frac{I^\alpha_{\kappa_0}}{(1-\alpha)}}{\kappa_0}\right\rfloor.\]
    Finally, taking the limit as $\alpha\rightarrow 1^-$, we get 
\[d(x,y) \leq \left\lfloor\frac{2+ I_{\kappa_0}}{\kappa_0}\right\rfloor.\]
Since $x$ and $y$ were arbitrary, the estimate on the diameter of $G$ follows. 
\end{proof}

\subsubsection{Examples}
The diameter estimate \eqref{LLY_diam_bound} is sharp, as the following example shows. This is also an example where the diameter estimate from \cite{LinLuYau} does not apply, since there are edges with curvature $\kappa_{LLY}=0$. 
\begin{example}\label{example_sharp_diam}
    The Path graph $G=P_n$ with $n\geq 3$ vertices has Lin-Lu-Yau Ricci curvature $1$ on the edges adjacent to leaves, and $0$ on all other edges. Therefore, choosing $\kappa_0 =1$, we have $I_{\kappa_0} = |E|-2 = (n-3)$, hence \eqref{LLY_diam_bound} gives
    \[{\rm Diam}(P_n) \leq \left\lfloor \frac{2+I_{\kappa_0}}{\kappa_0} \right\rfloor = 2+(n-3) = n-1,\]
    which is sharp.
    \begin{figure}[h]
        \centering
        \includegraphics[width=0.5\linewidth]{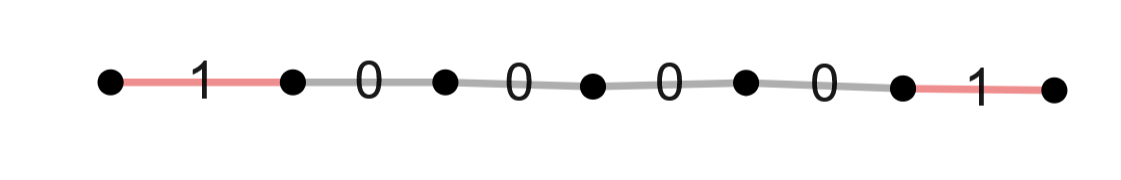}
        \caption{Representation of $P_7$ with Lin-Lu-Yau curvature labeled on each edge, generated using the Graph Curvature Calculator from \cite{CKLLS}.}
        \label{fig:enter-label}
    \end{figure}
\end{example}
Estimate \eqref{LLY_diam_bound} also holds when there is strictly negative curvature, like in the following example, contrasting with what happens with the diameter estimate from \cite{LinLuYau}.
\begin{example}
    For any $m\geq 3$, consider the Dumbbell Graph $G=K_m-K_m$ consisting of two copies of complete graphs $K_m$ with $m$ vertices, joined by one edge. Let $\{x_0,x_1,\ldots, x_{m-1}\}$ and $\{y_0,y_1,\ldots, y_{m-1}\}$ denote the vertices of each copy of $K_m$, respectively, and suppose w.l.o.g. that $x_0y_0\in E$. Let $i,j\not =0$ be distinct indices. We can distinguish between the following three kinds of edge curvatures
    \begin{align*}
        \kappa_{LLY}(x_i,x_j) = \kappa_{LLY}(y_i,y_j) &= \frac{m}{m-1},\\
        \kappa_{LLY}(x_0,x_j) = \kappa_{LLY}(y_0,y_j) &= \frac{(m-1)^2+1}{m(m-1)},\\
        \kappa_{LLY}(x_0,y_0) &= -\frac{2(m-2)}{m}.
    \end{align*}
Choosing $\kappa_0 = \displaystyle \frac{(m-1)^2+1}{m(m-1)}$, we have
\[I_{\kappa_0} = \frac{(m-1)^2+1}{m(m-1)} + \frac{2(m-2)}{m} = \frac{3m^2-8m+6}{m(m-1)}.\]
Hence, \eqref{LLY_diam_bound} gives the estimate
\[{\rm Diam}(K_m-K_m) \leq \left\lfloor \frac{2+\frac{3m^2-8m+6}{m(m-1)}}{\frac{(m-1)^2+1}{m(m-1)}}\right\rfloor = \left\lfloor \frac{5(m-1)^2+1}{(m-1)^2+1} \right\rfloor =\left\lfloor 5 -\frac{4}{(m-1)^2+1} \right\rfloor=4.\]
Note that, although the estimate is not sharp in this case, since ${\rm Diam}(K_m-K_m)=3$, it is a very close bound and, most notably, it is $O(1)$, i.e. independent of $m$. This is remarkable, as the diameter of a connected graph with $n=2m$ vertices could be as large as $2m-1$ in the case of $P_{2m}$. 
    \begin{figure}[h]
        \centering
        \includegraphics[width=0.5\linewidth]{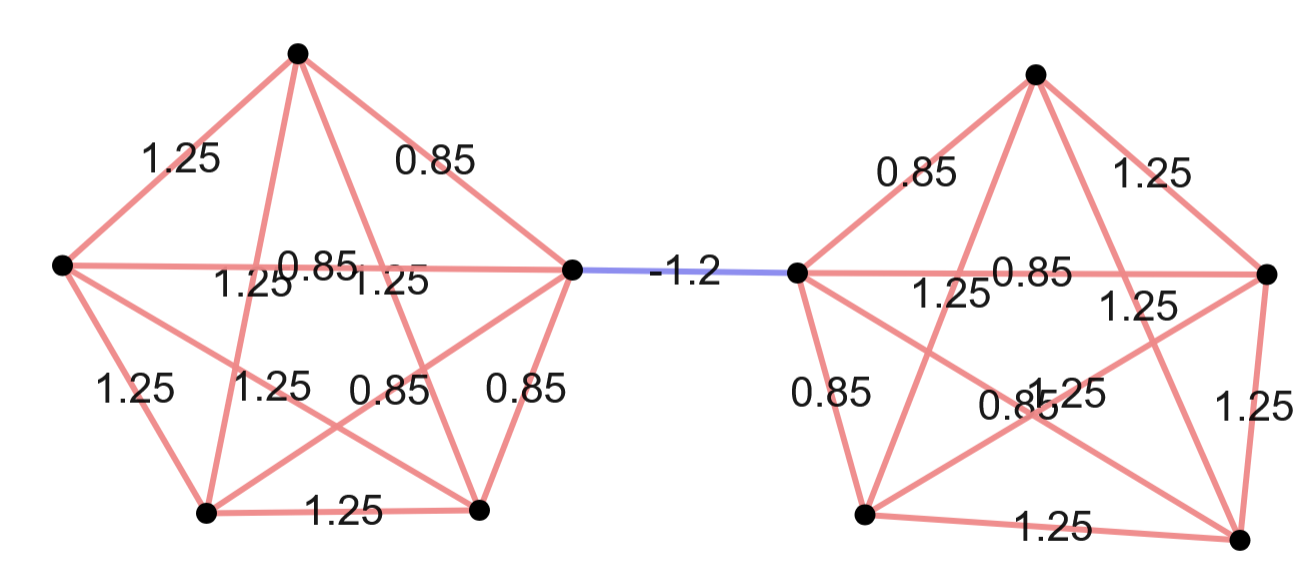}
        \caption{Representation of $K_5-K_5$ with Lin-Lu-Yau curvature labeled on each edge, generated using the Graph Curvature Calculator from \cite{CKLLS}.}
        \label{fig:enter-label}
    \end{figure}
\end{example}

Similarly to what happens in the manifold case, estimates with integral curvature conditions are only meaningful when there is a small amount of curvature below the threshold $\kappa_0$, i.e. when the graph is a small perturbation of a graph with Ricci curvature bounded below by $\kappa_0$. One can see this in the following example, where the diameter bound \eqref{LLY_diam_bound} becomes meaningless, since it reaches the maximum diameter of any connected graph with $n$ vertices, ${\rm Diam}(G)\leq n-1$, which is only achieved by the path graph $P_n$. 

\begin{example}
    For any $m\geq 0$, consider a symmetric binary tree with $2^{m+1}-1$  nodes constructed in the following way. For $0\leq i\leq m$, arrange the nodes in $m+1$ layers, with $2^i$ nodes in the $i$-th layer, and connect each of them to one of the nodes in the $(i-1)$-th layer, except for the node in layer $0$; each node in the $i$-th layer should be connected to exactly two nodes in the $i+1$-th layer, except for the nodes in layer $m$. Attach a leaf at the only node on layer $0$, and denote the resulting binary tree $T_m$ (see Figure~\ref{tree_fig} below), which has a total of $n=|V|=2^{m+1}$ vertices, $|E| = 2^{m+1}-1$ edges, and a diameter of ${\rm Diam}(T_m) = \max\{2m,1\}$.
    \vspace{1em}

    $T_m$ has two kinds of edges: the ones adjacent to a leaf have curvature $\kappa_{LLY}=2/3$, and the ones not adjacents to a leaf have curvature $\kappa_{LLY}=-2/3$. Using $\kappa_0 = 2/3$ and the fact that the number of edges adjacent to a leaf is $2^m+1$, we have 
    \[I_{\kappa_0} = \frac{4}{3}\left( |E| - (2^m+1)\right) = \frac{4}{3}\left( 2^{m+1}-1-(2^m+1)\right)=\frac{8}{3}(2^{m-1}-1).\]
    Therefore, 
    \[{\rm Diam}(T_m) \leq \left\lfloor \frac{2+I_{\kappa_0}}{\kappa_0} \right\rfloor\leq \left\lfloor \frac{2+\frac{8}{3}(2^{m-1}-1)}{2/3} \right\rfloor=2^{m+1}-1.\]
    As stated above, this estimate is meaningless in the sense that, despite knowing the curvatures of all the edges in the graph, the estimate obtained from \eqref{LLY_diam_bound} is the same as the largest diameter that a connected graph with $2^{m+1}$ vertices could have, which we could have known without studying its Lin-Lu-Yau Ricci curvature. Moreover, notice that the estimate obtained above is much larger than the actual diameter of $T_m$, which is $2m$ when $m\geq 1$. Intuitively, this is due to the fact that $I_{\kappa_0}$ contains information about all edges in $T_m$ with curvature below $\kappa_0$, as opposed to looking only at the edges of the path realizing the diameter of the graph. In short, there is too much curvature below $\kappa_0$ to obtain a meaningful estimate.

 \begin{figure}[h]
        \centering
        \includegraphics[width=0.5\linewidth]{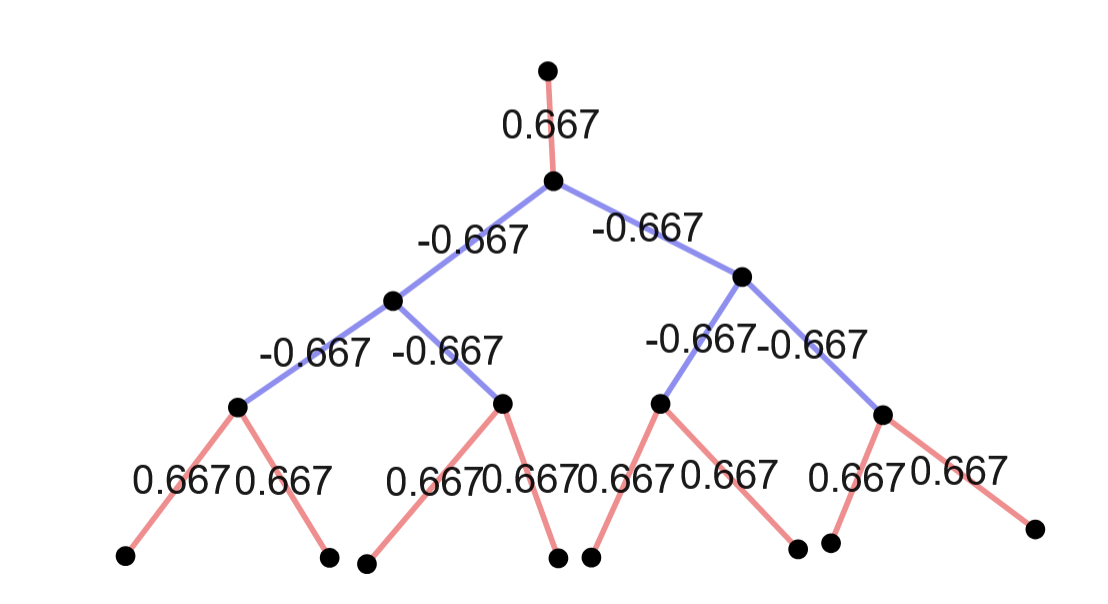}
        \caption{Representation of $T_3$ with Lin-Lu-Yau curvature labeled on each edge, generated using the Graph Curvature Calculator from \cite{CKLLS}.}
        \label{tree_fig}
    \end{figure}

    In some cases, \eqref{LLY_diam_bound} may even exceed the trivial bound on the diameter ${\rm Diam}(G)\leq n-1$ of a connected graph with $|V|=n$. This is the case of the graph $G$ in Figure~\ref{benzene_fig} below, with $n=12$ vertices, when one chooses $\kappa_0 = 2/3$, which leads to $I_{\kappa_0}=8$ and the bound 
    \[{\rm Diam}(G) \leq \left\lfloor \frac{2+8}{2/3}\right\rfloor =15,\]
    which exceeds $n-1=11$.
     \begin{figure}[h]
        \centering
        \includegraphics[width=0.5\linewidth]{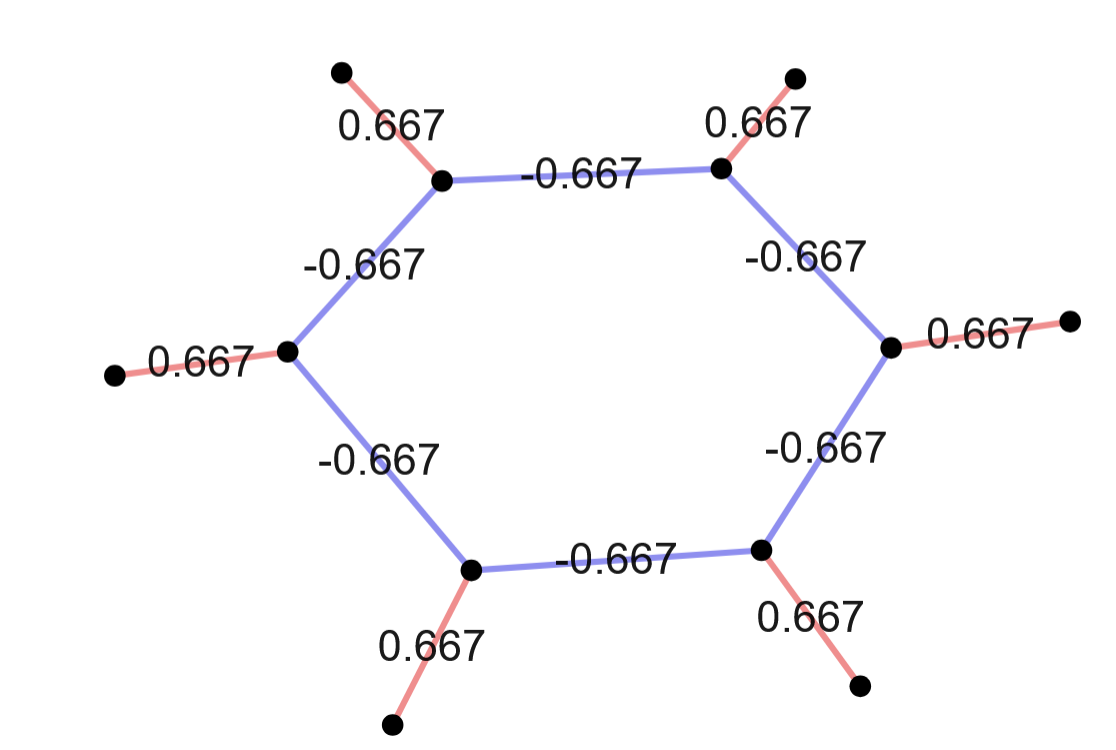}
        \caption{Representation of $G$ with Lin-Lu-Yau curvature labeled on each edge, generated using the Graph Curvature Calculator from \cite{CKLLS}.}
        \label{benzene_fig}
    \end{figure}
\end{example}

\subsection{Moore-type bound}

 Recall that Moore's bound establishes that a graph with maximum degree $d_M$ and diameter bounded by $D$ will have a number of vertices $n=|V|$ of at most 
\begin{equation}\label{Moore}
    n\leq 1+\sum_{k=1}^D d_M(d_M-1)^{k-1}.
\end{equation}
Theorem~\ref{LLY_Moore} is a refined version of Moore's estimate for positively curved graphs. In this subsection we will generalize Theorem \ref{LLY_Moore} to the integral curvature setting, proving Theorem~\ref{Moore_intcurv}, which we restate below for convenience.

\begin{theorem*}[\bf\ref{Moore_intcurv}]
Let $d_M$ be the maximum degree of $G$, and suppose that ${\rm Diam}(G)=D$. For any $\kappa_0\in \mathbb{R}$, the number of vertices $n=|V|$ of $G$ is at most

\begin{equation}\label{Moore_intcurv_bound_diam}
    n\leq 1+ \sum_{k=1}^{D} (d_M)^k \prod_{i=1}^{k-1}\left[1+ \frac{I_{\kappa_0}-i\kappa_0}{2}\right].
\end{equation}
Moreover, if $\kappa_0>0$,  we have
\begin{equation}\label{Moore_intcurv_bound}
    n\leq 1+ \sum_{k=1}^{\left\lfloor \frac{2+I_{\kappa_0}}{\kappa_0}\right\rfloor} (d_M)^k \prod_{i=1}^{k-1}\left[1+ \frac{I_{\kappa_0}-i\kappa_0}{2}\right].
\end{equation}
 
\end{theorem*}
\begin{remark} Note that, by contrast to Moore's result \eqref{Moore},  Theorem \ref{LLY_Moore} and the estimate \eqref{Moore_intcurv_bound} in Theorem \ref{Moore_intcurv} do not require a priori bounds on the diameter of the graph, since those follow from Theorems \ref{LLY_BonnetMyers} and \ref{BonnetMyers_intcurv}, respectively. Moreover, Theorem \ref{Moore_intcurv} does not require a positive lower bound on the curvature. 
\end{remark}
\begin{remark} For large $d_M$, estimate \eqref{Moore_intcurv_bound_diam} is a significant improvement over Moore's estimate \eqref{Moore} when $I_{\kappa_0}-\kappa_0<0$. Estimate \eqref{Moore_intcurv_bound} recovers Theorem~\ref{LLY_Moore} in the limit where $\kappa_{LLY}(x,y)\geq \kappa_0>0$ for all edges $xy\in E$.
\end{remark}
To establish this result, the following notation will be useful. Let $x,y\in V$ be two distinct vertices. The neighborhood of $y$, $\Gamma(y)$, can be partitioned into the following sets
\begin{align*}
    \Gamma_x^+(y) &:= \left\{ v\in \Gamma(y) \colon \ d(x,v) = d(x,y)+1\right\},\\
    \Gamma_x^0(y) &:= \left\{ v\in \Gamma(y) \colon \ d(x,v) = d(x,y)\right\},\ \text{and}\\
    \Gamma_x^-(y) &:= \left\{ v\in \Gamma(y) \colon \ d(x,v) = d(x,y)-1\right\}.
\end{align*}
We will also need the following Lemma from \cite{LinLuYau}.
\begin{lemma}[\cite{LinLuYau} Lemma~4.4]\label{LLY_lemma4.4} For any two distinct vertices $x$ and $y$, we have
\begin{equation}
    \kappa_{LLY}(x,y)\leq \displaystyle\frac{\displaystyle 1+\frac{|\Gamma_x^-(y)|-|\Gamma_x^+(y)|}{d_y}}{d(x,y)}.
\end{equation}
\end{lemma}

\begin{proof}[Proof of Theorem~\ref{Moore_intcurv}]
This proof will follow a similar argument to the proof of Theorem 4.3 in \cite{LinLuYau}. For any distance $i$, $1\leq i\leq D$, define
\[\Gamma_i(x):=\left\{v\in G\colon d(x,v)=i\right\}.\]
For any $y\in \Gamma_i(x)$, using Lemma~\ref{keylemma} and Lemma~\ref{LLY_lemma4.4}, we have
\begin{align*}
2+\left[I_{\kappa_0}-\kappa_0 i\right]&\geq 2-d(x,y)\kappa_{LLY}(x,y)\\
&\geq 2-\left[1+\frac{|\Gamma_x^-(y)|-|\Gamma_x^+(y)|}{d_y}\right]\\
&= \frac{d_y-|\Gamma_x^-(y)|+|\Gamma_x^+(y)|}{d_y}.    
\end{align*}
Since $d_y=|\Gamma(y)|=|\Gamma_x^+(y)|+|\Gamma_x^0(y)|+|\Gamma_x^-(y)|$, using that $|\Gamma_x^0(y)|\geq 0$, we obtain
\[2+I_{\kappa_0}-\kappa_0 i \geq 2\frac{|\Gamma_x^+(y)|}{d_y}.\]
Therefore, 
\[|\Gamma_x^+(y)|\leq \left[1+ \frac{I_{\kappa_0}-i\kappa_0}{2} \right]d_y\leq\left[1+ \frac{I_{\kappa_0}-i\kappa_0}{2} \right]d_M.\]
Notice that if $v\in \Gamma_{i+1}(x)$, then $v\in \Gamma_x^+(y)$ for some $y\in \Gamma_i(x)$. Thus, we have
\[|\Gamma_{i+1}(x)|\leq \sum_{y\in \Gamma_i(x)}|\Gamma_x^+(y)| \leq \sum_{y\in \Gamma_i(x)}\left[ 1+\frac{I_{\kappa_0}-i\kappa_0}{2}\right]d_M \leq |\Gamma_i(x)|\left[ 1+\frac{I_{\kappa_0}-i\kappa_0}{2}\right]d_M .\]
Since $|\Gamma_1(x)| = d_x\leq d_M$, by induction over $k$, we have 
\[|\Gamma_k(x)|\leq (d_M)^k \prod_{i=1}^{k-1}\left[1+\frac{I_{\kappa_0}-i\kappa_0}{2}\right].\]
Finally, by exhausting all the possible distances, we have 
\begin{align*}
    n&=1+\sum_{k=1}^{D} |\Gamma_k(x)|\\
    & \leq 1+\sum_{k=1}^{D} (d_M)^k\prod_{i=1}^{k-1}\left[1+ \frac{I_{\kappa_0}-i\kappa_0}{2}\right].
\end{align*}
Moreover, if $\kappa_0>0$, then by Theorem~\ref{BonnetMyers_intcurv} we know that  $D={\rm Diam}(G)\leq \left\lfloor \frac{2+I_{\kappa_0}}{\kappa_0} \right\rfloor$. Using $D=\left\lfloor \frac{2+I_{\kappa_0}}{\kappa_0} \right\rfloor$ we obtain \eqref{Moore_intcurv_bound}.
\end{proof}

\begin{remark}\label{remark_Moore_graphs}
    One may hope that a result like Theorem~\ref{Moore_intcurv} would give some curvature obstruction for the existence of a Moore graph $G$ with $d_M=57$ and ${\rm Diam}(G)=2$, whose existence, to the best of our knowledge, is still unknown  (see the survey \cite{Dalfo}, as well as \cite{BannaiIto}, \cite{Damerell}, \cite{FaberKeegan}). A Moore graph is a graph satisfying equality in estimate \eqref{Moore}. Therefore, a Moore graph with maximum degree $d_M$ and diameter $D=2$ will have  $n=1+(d_M)^2$. Using \eqref{Moore_intcurv_bound_diam} one can get
    \[\kappa_0-I_{\kappa_0}\leq \frac{2}{d_M}.\]
    In particular, it will have at least one edge satisfying
    \[\kappa_{LLY}(x,y)\leq \frac{2}{d_M}.\]
    Unfortunately, these are not new obstructions when $d_M=57$, since it follows from Theorem~1.4 in \cite{LinYou} that for $d_M>6$ such graph would not be positively curved, which implies $\kappa_0-I_{\kappa_0}<0$.
\end{remark}



    

\subsection{Lichnerowicz-type estimate} 

Denote $\lambda_1$ the first non-zero eigenvalue of the graph Laplacian \[\Delta = I - D^{-1}A,\]
where $A$ denotes the adjacency matrix of $G$ and $D$ is the diagonal matrix of degrees.

\begin{theorem*}[\bf \ref{Lichnerowicz_intcurv}]
For any $\kappa_0>0$ and any $\alpha \in [0,1)$, we have
\begin{equation}\label{Lic_int_alpharicci}
    \lambda_1 \geq \kappa_0 -\frac{I_{\kappa_0}^\alpha}{1-\alpha},
\end{equation}
and
\begin{equation}\label{Lic_int_LLYricci}
    \lambda_1 \geq \kappa_0 - I_{\kappa_0}.
\end{equation}
\end{theorem*}
\begin{proof}
    As in the proof of Theorem~4.2 in \cite{LinLuYau}, let $M_\alpha$ be the average operator associated to the $\alpha$-lazy random walk, i.e. for any function $f:V\rightarrow \mathbb{R}$, $M_\alpha[f]$ is the function defined as 
    \begin{equation}
        M_\alpha[f](x) = \sum_{z\in V} f(z) m_x^\alpha(z),
    \end{equation}
    for any $x\in V$.
    \vspace{1em}
    
    It follows from Lemma~\ref{keylemma} that 
    \begin{equation}
    W(m_x^\alpha, m_y^\alpha) = \left[ 1-\kappa_\alpha(x,y)\right] d(x,y)\leq \left[ 1+(1-\alpha)\left( \frac{I_{\kappa_0}^\alpha}{(1-\alpha)d(x,y)}-\kappa_0\right)\right] d(x,y),  
    \end{equation}
    therefore if $x\not = y$, $d(x,y)\geq 1$, and we have 
    \begin{equation}
        W(m_x^\alpha, m_y^\alpha) \leq \left[ 1+(1-\alpha)\left( \frac{I_{\kappa_0}^\alpha}{(1-\alpha)}-\kappa_0\right)\right] d(x,y).
    \end{equation}
    
    Using the dual expression \eqref{transp_distance_dual} of the transportation distance, if $f$ is $1$-Lipschitz, we have
    \begin{align*}
        \left|M_\alpha[f](x)-M_\alpha[f](y)\right| & = \left| \sum_{z\in V} f(z) (m_x^\alpha(z)-m_y^\alpha(z))\right|\\
        &\leq W(m_x^\alpha,m_y^\alpha)\\
        &\leq \left[ 1+(1-\alpha)\left( \frac{I_{\kappa_0}^\alpha}{(1-\alpha)}-\kappa_0\right)\right] d(x,y).
    \end{align*}
    Therefore, $M_\alpha[f]$ is $\left[ 1+(1-\alpha)\left( \frac{I_{\kappa_0}^\alpha}{(1-\alpha)}-\kappa_0\right)\right]$-Lipschitz, and the mixing rate of $M_\alpha$ is at most $1+(1-\alpha)\left( \frac{I_{\kappa_0}^\alpha}{(1-\alpha)}-\kappa_0\right)$. On the other hand, $M_\alpha$ can be written as 
    \[M_\alpha = \alpha I + (1-\alpha)D^{-1}A = I-(1-\alpha)\Delta,\]
    hence its eigenvalues are $1$, $1-(1-\alpha)\lambda_1$, $1-(1-\alpha)\lambda_2$, $\ldots$, $1-(1-\alpha)\lambda_{n-1}$. Therefore, the mixing rate of $M_\alpha$ is exactly $1-(1-\alpha)\lambda_1$, so we have 
    \[1-(1-\alpha)\lambda_1 \leq 1+(1-\alpha)\left( \frac{I_{\kappa_0}^\alpha}{(1-\alpha)}-\kappa_0\right),\]
    or equivalently,
    \[\lambda_1 \geq \kappa_0 - \frac{I_{\kappa_0}^\alpha}{(1-\alpha)}.\]
    Hence, taking the limit as $\alpha \rightarrow 1^{-}$, we get
    \[\lambda_1 \geq \kappa_0 - I_{\kappa_0}.\]
\end{proof}
\begin{remark}
    Note that estimate \eqref{Lic_int_LLYricci} recovers \eqref{Lichnerowicz_estimate} in the limit where $\kappa_{LLY}(x,y)\geq \kappa_0$ for all edges $xy\in E$. 
\end{remark}
\begin{remark}\label{remark_Lichnerowicz}
    In contrast to what happens in \cite{Aubry} for the smooth setting, due to the discrete nature of graphs, this estimate is not an improvement on the pointwise result of \cite{LinLuYau}. One can see that if there is a single edge with curvature $\kappa_1<\kappa_0$, then $I_{\kappa_0} = \kappa_0-\kappa_1$, and the estimate \eqref{Lic_int_LLYricci} will read 
    \[\lambda_1\geq \kappa_1.\]
    In particular, in contrast to what happens for Theorem~\ref{BonnetMyers_intcurv} and Theorem~\ref{Moore_intcurv}, this estimate does not provide any information if the graph is not positively curved. Despite this, we are including this result as a reformulation of Theorem~4.2 in \cite{LinLuYau}, which may be of interest in a context where one knows how to control $I_{\kappa_0}$ but doesn't know the exact curvatures on the graph.
\end{remark}


\end{document}